\documentclass[12pt]{amsart}

\usepackage{graphicx}
\usepackage{subfigure}
\usepackage{eepic}
\usepackage{xcolor}
\selectcolormodel{gray}
\usepackage{tikz}
\usepackage{amsmath}
\usepackage{a4wide}
\usepackage[utf8]{inputenc}
\usepackage{amssymb}
\usepackage{amsopn}
\usepackage{epsfig}
\usepackage{amsfonts}
\usepackage{latexsym}
\usepackage{amsthm}
\usepackage{enumerate}
\usepackage[UKenglish]{babel}
\usepackage{verbatim}
\usepackage{color}

\newtheorem{theorem}{Theorem}[section]
\newtheorem{lemma}[theorem]{Lemma}
\newtheorem{proposition}[theorem]{Proposition}
\newtheorem{corollary}[theorem]{Corollary}

\theoremstyle{definition}

\newtheorem{example}[theorem]{Example}

\theoremstyle{remark}
\newtheorem{remark}[theorem]{Remark}

\numberwithin{equation}{section}

\renewcommand{\i}{\mathbf{i}}

\providecommand{\pen}{\mathbf{p}_{\epsilon, n}}

\newcommand{\p}{\textbf{p}}

\begin{document}

\title[Maximising Bernoulli measures and dimension gaps]{Maximising Bernoulli measures and dimension gaps for countable branched systems}

\author{Simon Baker and Natalia Jurga}
\address{Simon Baker: Mathematics institute, University of Warwick, Coventry, CV4 7AL, UK}
\email{simonbaker412@gmail.com}
\address{Natalia Jurga: Mathematics institute, University of Warwick, Coventry, CV4 7AL, UK}
\email{N.Jurga@warwick.ac.uk }

\date{\today}

\subjclass[2010]{}

\begin{abstract}
Kifer, Peres, and Weiss proved in \cite{KPW} that there exists $c_0>0,$ such that $\dim \mu\leq 1-c_0$ for any probability measure $\mu$ which makes the digits of the continued fraction expansion i.i.d. random variables. In this paper we prove that amongst this class of measures, there exists one whose dimension is maximal. Our results also apply in the more general setting of countable branched systems.
\end{abstract}

\keywords{Continued fractions, Bernoulli measures, Dimensions of measures.}
\maketitle

\section{Introduction}\label{sec:1}
Let $x\in[0,1]\setminus \mathbb{Q}.$ Then as is well known, there exists a unique sequence $(a_i)\in\mathbb{N}^{\mathbb{N}}$ such that $$x=\cfrac{1}{a_1+\cfrac{1}{a_2+\cfrac{1}{a_3+\ldots}}}.$$ The sequence $(a_i)$ is called the continued fraction expansion of $x$. We can generate $(a_i)$ using the Gauss map $T:[0,1]\setminus \mathbb{Q}\to [0,1]\setminus \mathbb{Q},$ which is defined to be $$T(x)=\frac{1}{x}\, \textrm{(mod }1).$$ The sequence $(a_i)$ is then constructed via the rule $$a_i=\Big\lfloor{\frac{1}{T^{i-1}(x)}}\Big\rfloor.$$ Where $\lfloor \cdot\rfloor$ denotes the integer part. One can study the  statistical properties of $T$ using the Gauss measure $\mu_{G}$, which is given by $$\mu_{G}(A)=\frac{1}{\log 2}\int_{A} \frac{1}{1+x}\, \mathrm{d} x$$ for any Borel subset $A\subset [0,1]$. The measure $\mu_{G}$ is $T$-invariant and ergodic. Importantly $\mu_{G}$ is also absolutely continuous with respect to the Lebesgue measure. Consequently one can use $\mu_{G}$ to derive statistical information about the sequence $(a_i)$ for Lebesgue almost every $x$.

 Using the shift space $(\mathbb{N}^{\mathbb{N}},\sigma)$ one can ``code'' the dynamics of $T$. Let $\Pi:\mathbb{N}^{\mathbb{N}}\to [0,1]\setminus \mathbb{Q}$ be the map satisfying $$\Pi((a_i)):=\cfrac{1}{a_1+\cfrac{1}{a_2+\cfrac{1}{a_3+\ldots}}}.$$ Then $T\circ \Pi= \Pi \circ \sigma$. Where $\sigma$ is the usual shift map. One can define many $T$-invariant measures on $[0,1]\setminus \mathbb{Q}$ using the coding map $\Pi$. Indeed, for any $\sigma$-invariant measure $m$, one can define a $T$-invariant measure $\Pi_* m= m \circ \Pi^{-1}.$ The fact that $\Pi_* m$ is $T$-invariant follows from the relation $T\circ \Pi= \Pi \circ \sigma$. We call $\Pi_* m$ the pushforward of $m$. The simplest $\sigma$-invariant measures on $\mathbb{N}^{\mathbb{N}}$ are the Bernoulli measures $m_{\textbf{p}}$ corresponding to a probability vector $\textbf{p}=(p_i)_{i=1}^{\infty}$. In what follows we let $\mu_{\p}:=\Pi_* m_\p.$ From a statistical perspective, it would be highly desirable for the pushforward of a Bernoulli measure to be absolutely continuous with respect to the Lebesgue measure. This is unfortunately not the case for $T$, and we are forced to realign our expectations and hope that there exists a ``large'' set whose dynamics can be described by the pushforward of a Bernoulli measure. For us large will be described by the dimension of a measure. For an arbitrary Borel probability measure $\mu$ supported on $[0,1],$ we define the dimension of $\mu$ to be $$\dim (\mu):=\inf\{ \dim_{H}(A): \mu(A)=1\}.$$ One can prove using the thermodynamic formalism developed by Walters \cite{Wal}, and a result of Kinney and Pitcher \cite{KP}, that whenever $-\sum p_i \log p_i<\infty$ we have
\begin{equation}
\label{Walters}
\dim \mu_{\textbf{p}}<1.
\end{equation}  What is not clear from \eqref{Walters} is whether $\dim  \mu_{\textbf{p}}$ can be arbitrarily close to $1$. This problem is difficult since $\dim(\cdot)$ is not necessarily upper semi-continuous as a real valued function on the space of $T$-invariant probability measures equipped with the weak star topology, and the set $\{\mu_\p\}$ is not compact. That being said, an answer to this question was obtained in a paper of Kifer, Peres, and Weiss \cite{KPW}, who proved the following theorem.
\begin{theorem}
	\label{KPW theorem}
$\sup_{\p}\dim \mu_{\p}<1-10^{-7}.$
\end{theorem} Theorem \ref{KPW theorem} was recently extended by Rapaport \cite{Rap} who showed that there exists $c_0>0,$ such that whenever $\mu$ is a measure which makes the digits independent, not necessarily identically distributed, then $\dim \mu <1-c_0$. In an upcoming paper of the second author \cite{Jur}, another proof is given that there exists a $c_0>0$, which can be made explicit, such that $\sup_{\p}\dim \mu_{\p}<1-c_0.$ This proof makes use of techniques from thermodynamic formalism.

The main result of this paper gives a new proof that $\sup_{\p}\dim  \mu_{\p}<1$. Our result is weaker than Theorem \ref{KPW theorem} in the sense that we do not achieve any quantitative information on the size of the \emph{dimension gap}. However, we improve upon Theorem \ref{KPW theorem} by showing that there exists a Bernoulli measure whose pushforward achieves the supremum. 

\begin{theorem}
	\label{supremum cf}
There exists a probability vector $\textbf{p}^*$ such that $\sup_{\textbf{p}}\dim \mu_{\textbf{p}}=\dim  \mu_{\textbf{p}^*}.$
\end{theorem}It will follow from our proof that the measure $\mu_{\textbf{p}^*}$ appearing in Theorem \ref{supremum cf} satisfies $-\sum p_i^* \log p_i^*<\infty.$ Applying \eqref{Walters}, we see that Theorem \ref{supremum cf} immediately implies $\sup_{\textbf{p}}\dim \mu_{\textbf{p}}<1$. As we will see, Theorem \ref{supremum cf} in fact holds more generally for countable branched systems that satisfy some regularity assumptions. 

Our proof of Theorem \ref{supremum cf} differs significantly from the approaches given in \cite{KPW} and \cite{Rap}. The approach of both of these papers relied on showing that for any of the considered measures $\mu,$ a generic point for $\mu$ is contained in a set exhibiting exceptional large deviation asymptotics. Importantly this exceptional set has no dependence on $\mu.$ Their problem then reduces to determining an upper bound for the Hausdorff dimension of this exceptional set. Our proof relies on studying those Bernoulli measures whose pushforward is supported on only the first $L$ digits. Restricting to this class of measures, it is known that there exists a measure $\mu_{\p^L}$ whose dimension is maximal. We will show that this measure always satisfies a certain decay property. This decay property allows us to rewrite $\dim \mu_{\p^L}$ as an expression involving finitely many digits up to some uniformly small error. Taking a weak star limit along some subsequence of $(\mu_{\p^L})_{L=1}^{\infty}$, we can then use this expression for $\dim  \mu_{\textbf{p}^L}$ to show that the limiting measure in fact achieves the supremum in Theorem \ref{supremum cf}. 

The rest of this paper is arranged as follows. In Section $2$ we recall some background from countable branched systems and state Theorem \ref{Main theorem}. In Section $3$ we prove Theorem \ref{Main theorem}. In Section $4$ we show how Theorem \ref{Main theorem} implies Theorem \ref{supremum cf} and make some concluding remarks.

\section{Preliminaries}
Let $\{I_n=(a_n,b_n)\}_{n=1}^{\infty}$ be a countable collection of disjoint open subintervals of $(0,1)$ such that either $a_1=0$, $b_n=a_{n+1}$ for all $n\in \mathbb{N},$ and $\lim_{n\to\infty} b_n=1$, or  $b_1=1,$ $b_{n+1}=a_{n}$ for all $n\in \mathbb{N},$ and $\lim_{n\to\infty} a_n=0.$   Assume that for each $I_n$ there exists a map $T_n:I_n\to (0,1)$ such that $T_n$ is a $C^2$ bijection from $I_n$ onto $(0,1)$. In what follows we always assume that every $T_n$ is  orientation preserving, or every $T_n$ is orientation reversing. We can then define the orientation preserving or orientation reversing map $T:\cup_{n=1}^{\infty} I_n \to (0,1)$ via the rule $T(x)=T_n(x)$ if $x\in I_n$. Throughout we will assume that $T$ satisfies the following conditions:

\begin{enumerate}
	\item (Monotone derivative). The derivative $T'$ is monotone on $\cup_{n=1}^{\infty} I_n.$ 
	\item (Orientation reversing) If $T$ is orientation reversing and $T'$ is increasing, then for any $n\in\mathbb{N}$ and $x,y\in I_n$ such that $x\geq y,$ we have $(T^2)'(x)\geq (T^2)'(y).$ If $T$ is orientation reversing and $T'$ is decreasing, then for any $n\in\mathbb{N}$ and $x,y\in I_n$ such that $x\geq y,$ we have $(T^2)'(x)\leq (T^2)'(y).$  
	\item (Uniformly expanding). Some iterate of $T$ is uniformly expanding, that is, there exists $l\in\mathbb{N}$ and $\Lambda>1$ such that $$|(T^l)'(x)|\geq \Lambda$$ for all $x\in [0,1]$.
	\item (Renyi condition). There exists $\kappa< \infty$ such that
	$$\sup_{n\in\mathbb{N}}\sup_{x,y,z\in I_n} \big| \frac{T''(x)}{T'(y)T'(z)}\big|\leq \kappa.$$
	\item  There exist $s\in(0,1)$ such that $$\sum_{n=1}^{\infty}|I_n|^s<\infty.$$
\end{enumerate} 
Let us emphasise here that the Gauss map satisfies $(1)-(5)$. Conditions $(1),(3),(4),$ and $(5)$ are standard assumptions. We expect Theorem \ref{Main theorem} holds without assuming condition $(2)$. We define
\begin{eqnarray}
s_0:=\inf \left\{s:\sum_{n=1}^{\infty}|I_n|^s<\infty\right\}.
\label{s0}
\end{eqnarray}
By $(5)$ we know that $s_0<1$. 

Let $\phi_n:(0,1)\to I_n$ be the inverse map of $T_n.$ Given a sequence $(a_i)\in\mathbb{N}^n$ we let $$\phi_{a_1,\ldots,a_n}:=\phi_{a_1}\circ\cdots \circ \phi_{a_n}\textrm{ and }I_{a_1,\ldots,a_n}:=\phi_{a_1,\ldots,a_n}((0,1))).$$  Under our assumptions we can code the dynamics of $T$ using the coding map $\Pi:\mathbb{N}^{\mathbb{N}}\to [0,1]$ defined as follows: $$\Pi((a_i)):=\bigcap_{n=1}^{\infty}\overline{I_{a_1,\ldots,a_n}}.$$  Notice that we again have the relation $\Pi \circ \sigma = T \circ \Pi$. We can use the map $\Pi$ to pushforward Bernoulli measures and again ask what is their dimension. Our main result is the following.

\begin{theorem}
	\label{Main theorem}
	Assume $T$ satisfies properties $(1)-(5)$. Suppose there exists $\mu_{\textbf{p}}$ such that $\dim \mu_{\textbf{p}}>s_0,$ then there exists $\mu_{\textbf{p}^*}$ such that $\sup_{\textbf{p}}\dim  \mu_{\textbf{p}}=\dim  \mu_{\textbf{p}^*}.$
\end{theorem}For the Gauss map it can be shown that $s_0=1/2$. So by Theorem \ref{Main theorem} to prove Theorem \ref{supremum cf} it suffices to construct a Bernoulli measure $\mu_{\textbf{p}}$ such that $\dim \mu_{\textbf{p}}>1/2$. We construct such a measure in Section $4$. Note that Theorem \ref{Main theorem} has the following straightforward corollary, which can be used to establish the existence of a dimension gap at $1$.

\begin{corollary}
	\label{gap corollary}
Assume $T$ satisfies properties $(1)-(5)$. Suppose that  $$\dim \mu_{\p}<1 \textrm{ for all }\mu_\p.$$ Then there exists some $c_0>0$ such that $$\sup_{\p}\dim \mu_{\p}\leq 1-c_0.$$
\end{corollary}Corollary \ref{gap corollary} follows since if we fail the hypothesis of Theorem \ref{Main theorem}, we must have $\dim \mu_{\p}\leq s_0$ for all $\mu_\p$ and $s_0<1$.

When studying the dimension of $T$-invariant measures the following dynamical quantities naturally arise. Given a map $T$ satisfying properties $(1)-(5)$ and a $T$-invariant measure $\mu$, we define the entropy of $\mu$ to be $$h(\mu):=\lim_{n\to\infty}\frac{1}{n}\sum_{(a_i)\in\mathbb{N}^n}-\mu(I_{a_1\ldots a_n})\log \mu(I_{a_1\ldots a_n}).$$ Note that when $\mu$ is the pushforward of a Bernoulli measure we have the following simpler expression for $h(\mu)$:$$h(\mu_\p)= -\sum_{i=1}^{\infty}p_i\log p_i.$$  We define the Lyapunov exponent of $\mu$ to be $$\chi(\mu):=\int \log|T'|\,\mathrm{d} \mu.$$ The following well known formula relates the dimension of $\mu$ to these two dynamical quantities, see \cite{MU} for a proof in our general setting, and \cite{KP} for a proof in the setting of the Gauss map.

\begin{proposition}
	\label{volume lemma}
Suppose $T$ satisfies properties $(1)-(5)$. If $\mu$ is an ergodic $T$-invariant measure and $h(\mu)<\infty,$ then $$\dim (\mu)= \frac{h(\mu)}{\chi(\mu)}.$$
\end{proposition} 
As a consequence of the Renyi condition we have the following useful bounded distortion property.

\begin{lemma}
	\label{bounded distortion}
	Suppose $T$ satisfies properties $(1)-(5)$. There exists a uniform constant $C>0$ such that for any finite word $(a_1, ..., a_n)$, 
	$$-C \leq \log \Big|\frac{ (T^n)'(x)}{(T^n)'(y)}\Big|\leq C$$
	for any $x,y \in I_{a_1,\ldots,a_n}$.
\end{lemma}

\section{Proof of Theorem \ref{Main theorem}}

When studying $\dim(\mu_\p)$ there are two cases that naturally arise. The case when $h(\mu_\p)$ is finite, and the case when $h(\mu_\p)$ is infinite. We start this section by obtaining a upper bound for $\dim (\mu_\p)$ when $h(\mu_\p)=\infty$.

\subsection{The case where $h(\mu_\p)=\infty$} In this section we prove the following proposition.

\begin{proposition}
	\label{infinite h}
	If $h(\mu_\p)=\infty$ then $\dim (\mu_\p)\leq s_0.$
\end{proposition}

We start by proving that $h(\mu_\p)=\infty$ implies $\chi(\mu_\p)=\infty$. Our proof is an adaptation of Lemma $3.1$ from \cite{FLM}.

\begin{lemma}
	\label{Ruelle Margulis}
	If $h(\mu_\p)=\infty$ then $\chi(\mu_\p)=\infty.$
\end{lemma}

\begin{proof}
We start by remarking that by Lemma \ref{bounded distortion} and the mean value theorem we have 
\begin{equation}
\label{bounda}\chi(\mu_\p)=\int \log |T'|\mathrm{d}\mu\geq  \sum_{n=1}^{\infty}\mu_\p(I_n)\log \frac{1}{|I_n|} -C
\end{equation} for some constant $C>0$. We observe that for any $N\in\mathbb{N}$ we have
\begin{align*}
-\sum_{n=1}^N\mu_{\p}(I_n)\log \mu_{\p}(I_n) + \sum_{n=1}^{N}\mu_\p(I_n)\log |I_n|& = \sum_{n=1}^{N}\mu_\p(I_n)\log \frac{|I_n|}{\mu_\p(I_n)}\\
&=\sum_{n=1}^{N}\mu_\p(I_n)\cdot \sum_{n=1}^{N}\frac{\mu_\p(I_n)}{\sum_{n=1}^{N}\mu_\p(I_n)}\log \frac{|I_n|}{\mu_\p(I_n)}\\
& \leq \sum_{n=1}^{N}\mu_\p(I_n)\cdot \log \Big( \sum_{n=1}^N \frac{|I_n|}{\sum_{n=1}^N\mu_\p(I_n)}\Big)
\end{align*}
In our final step we used Jensen's inequality and the fact that $\log$ is a concave function. Since $\sum_{n=1}^{\infty}|I_n|=1$ our upper bound converges to $0$ as $N\to \infty$ . It follows therefore that if $-\sum_{n=1}^{\infty}\mu_{\p}(I_n)\log \mu_{\p}(I_n)=\infty$ then $-\sum_{n=1}^{\infty}\mu_\p(I_n)\log |I_n|=\infty$. By \eqref{bounda} this implies that if $h(\mu_\p)=\infty$ then $\chi(\mu_\p)=\infty$.
\end{proof}

Let $J_n(x)=I_{a_1,\ldots,a_n}$ if $x\in I_{a_1,\ldots,a_n}.$ Given $\lambda>0,$ we let $$E_{\lambda}:=\bigcap_{j=1}^{\infty}\bigcup_{n=j}^{\infty}\big\{x\in(0,1):|J_{n}(x)|\leq e^{-\lambda n}\big\}.$$ Given $s\in(0,1)$ let $$K_T(s):=\sup_{x\in\bigcup_{n}I_{n}}\sum_{y:T(y)=x}|T'(y)|^{-s}.$$ Note that it follows from Lemma \ref{bounded distortion} that $\sum_{n=1}^{\infty}|I_n|^s<\infty$ if and only if $K_{T}(s)<\infty$.
The following theorem was established in \cite{KPW}.

\begin{theorem}
	\label{Lyapunov bound}
Assume that $T$ satisfies properties $(1)-(5)$. For $s\in(s_0,1)$ we have $$q(s):=\lim_{n\to\infty}\frac{1}{n}\log \int |(T^n)'(x)|^{1-s} \mathrm{d}x \leq K_{T}(s).$$ Moreover, for any $\lambda>0$ we have $$\dim_{H}(E_\lambda)\leq \inf_{s_0<s<1} \Big(s+\frac{q(s)}{\lambda}\Big).$$
\end{theorem}
Using Theorem \ref{Lyapunov bound} we may now prove Proposition \ref{infinite h}. 
\begin{proof}[Proof of Proposition \ref{infinite h}]
	Let $\mu_\p$ be such that $h(\mu_\p)=\infty$. By Lemma \ref{Ruelle Margulis} we must have $\chi(\mu_\p)=\infty$. It follows then from the Birkhoff ergodic theorem that for $\mu_\p$ a.e. $x$ we have $$\lim_{n\to\infty}\frac{1}{n}\sum_{j=0}^{n-1} \log |T'(T^j)(x)|=\infty.$$ By Lemma \ref{bounded distortion} and the chain rule, this implies that for any $\lambda>0,$ $\mu_\p$ a.e. $x$ is contained in $E_{\lambda}$. Therefore $\mu_\p$ gives full measure to $E_{\lambda}$ for any $\lambda>0$. Consequently $\dim \mu_\p \leq \dim_{H}(E_\lambda)$ for any $\lambda>0$. Applying Theorem \ref{Lyapunov bound} we may conclude that $\dim \mu_\p \leq s_0$.
\end{proof}

\subsection{The case where $h(\mu_\p)<\infty$}
We start this section by introducing some notation and proving two lemmas which describe how the entropy and the Lyapunov exponent of $\mu_\p$ change when mass is moved from the $n$-th coordinate to the first coordinate.

Given a probability vector $\p=(p_1, p_2, ...),$ for $n\geq 2$ and  $0\leq \epsilon \leq \min\{p_{n}, 1-p_1\}$ we let $\p_{\epsilon, n}$ denote the probability vector
$$\p_{\epsilon, n}= (p_1+\epsilon, p_2, ..., p_n-\epsilon, p_{n+1}, ...). $$ That is, $\p_{\epsilon, n}$ is the probability vector obtained from $\p$ when $\epsilon$ mass has been moved from the $n$th coordinate to the first coordinate.


\begin{lemma}
\label{entropy}
Let $\p$ be such that $h(\mu_{\p})<\infty$. Then $$\frac{d}{d\epsilon}(h(\mu_{\p})-h(\mu_{\pen}))= \log\left(\frac{p_1+\epsilon}{p_n-\epsilon}\right).$$
\end{lemma}

\begin{proof}
Fix a probability vector $\p$ such that $h(\mu_{\p})<\infty$. Then for $0\leq \epsilon \leq \min\{p_{n}, 1-p_1\}$ we have 
$$h(\mu_{\p})-h(\mu_{\pen})= -p_n\log p_n-p_1\log p_1 + (p_n-\epsilon)\log(p_n-\epsilon)+(p_1+\epsilon)\log(p_1+\epsilon).$$
Differentiating the right hand side of the above we obtain
\begin{align*}\frac{d}{d\epsilon}(h(\mu_{\p})-h(\mu_{\pen}))&=-\frac{p_n-\epsilon}{p_n-\epsilon} - \log (p_n-\epsilon)+\frac{p_1+\epsilon}{p_1 +\epsilon}+\log (p_1+\epsilon)\\
&=-1 - \log (p_n-\epsilon)+1+\log (p_1+\epsilon)\\ &=\log\left(\frac{p_1+\epsilon}{p_n-\epsilon}\right).
\end{align*}
\end{proof}Lemma \ref{entropy} tells us that for small values of $\epsilon,$ the change in entropy when we move $\epsilon$ mass from the $n$-th coordinate to the first coordinate is approximately $\epsilon \log\left(\frac{p_1+\epsilon}{p_n-\epsilon}\right).$ We now prove a complimentary statement for the Lyapunov exponent. In order to quantify how the Lyapunov exponents change under redistribution of measure, we need to estimate the difference  $\int \log|T^{\prime} |d\mu_{\p}- \int \log|T^{\prime}| d\mu_{\p_{\epsilon,n}}$. It will be easier to estimate this quantity by rewriting the integrals over a common measure. 

Fix a probability vector $\p$. Let $\Sigma_0= (\{0\} \cup \mathbb{N})^{\mathbb{N}}$ denote the space of sequences whose entries are either an element of the natural numbers, or equal to the extra digit zero. We equip $\Sigma_0$ with the shift map $\sigma_0: \Sigma_0 \to \Sigma_0$.

We define two projections $\Pi_1: \Sigma_0 \to \Sigma$ and $\Pi_2: \Sigma_0 \to \Sigma$ given by
$$
\Pi_1((b_i))=(a_i) \textrm{ where } \left\{  \begin{array}{ccc}
a_i=b_i & \textnormal{if} & b_i \neq 0 \\
a_i=1 & \textnormal{if} & b_i=0\\
\end{array}\right.
$$
and
$$
\Pi_2((b_i))=(a_i)\textrm{ where } \left\{  \begin{array}{ccc}
a_i=b_i & \textnormal{if} & b_i \neq 0 \\
a_i=n & \textnormal{if} & b_i=0.\\
\end{array}\right.
$$
Let $\nu$ be the Bernoulli measure on $\Sigma_0$ associated to the probability vector $(q_0, q_1, ...)=(\epsilon, p_1, ..., p_n-\epsilon, p_{n+1}, ...)$. We make here the important observation that $$\Pi_*(\Pi_{1*}(\nu))= \mu_{\pen}\textrm{ and }\Pi_*(\Pi_{2*}(\nu))= \mu_{\p}.$$

Finally, denote
\begin{eqnarray}
\tau_n= \inf_{x \in I_n} |T^{\prime}(x)|
\label{tau}
\end{eqnarray}

\begin{lemma}
Let $T$ satisfy properties $(1)-(5)$ and $\p$ be a probability vector. Then there exists a constant $0<\lambda<1,$ such that for all $n\in N$ and $\epsilon$ sufficiently small
$$\chi(\mu_{\p})-\chi(\mu_{\pen}) \geq \epsilon \log (\lambda \tau_n)$$
\label{lyapunov}
\end{lemma}
\begin{proof}Let us start by fixing $s>s_0$. We split our proof into two cases: when $T$ is orientation preserving and when $T$ is orientation reversing. In both cases we will assume $T'$ is increasing. The case where $T'$ is decreasing is handled similarly. Note that under the assumption $T'$ is increasing and $T$ is orientation preserving the defining intervals $\{I_n=(a_n,b_n)\}\}$ must satisfy $a_1=0$ and $b_{n}\to 1$. Likewise if $T'$ is increasing and $T$ is orientation reversing the defining intervals $\{I_n=(a_n,b_n)\}$ must satisfy $b_1=1$ and $a_{n}\to 0.$
	\\

\textbf{Case $1$ ($T$ is orientation preserving).}
Let $\Pi_1$, $\Pi_2$ and $\nu$ be as above. Since the branches of $T$ are orientation preserving and $T'$ is increasing we have 
\begin{equation}
\label{projection increase}
\Pi\circ \Pi_2((b_i)) \geq \Pi \circ \Pi_1((b_i))
\end{equation}
for all $(b_i) \in \Sigma_0$. Given a finite word $\mathbf{j}=(j_1,\ldots,j_k)\in\{\{0\}\cup\mathbb{N}\}^k,$ we define the cylinder set determined by $\mathbf{j}$ to be $$[\mathbf{j}]:=\{(b_i)\in \Sigma_0:(b_1,\ldots,b_k) =\mathbf{j}\}.$$We observe
\begin{align}
\label{lyap1}
\int \log|T^{\prime} |d\mu_{\p}- \int \log|T^{\prime}| d\mu_{\p_{\epsilon,n}} &= \int \log |T^{\prime} \circ \Pi \circ \Pi_2| - \log |T^{\prime} \circ \Pi \circ \Pi_1| d\nu \nonumber\\ 
&=\sum_{i=0}^{\infty}  \int_{[i]} \log\left|\frac{T^{\prime} \circ \Pi \circ \Pi_2}{T^{\prime} \circ \Pi \circ \Pi_1}\right|d\nu \nonumber\\
&\geq  \int_{[0]} \log\left|\frac{T^{\prime} \circ \Pi \circ \Pi_2}{T^{\prime} \circ \Pi \circ \Pi_1}\right|d\nu.
\end{align}In the final inequality we used the fact that $T^{\prime}$ is increasing and \eqref{projection increase}. Since the absolute value of the derivative of any point in $I_1$ can be bounded above by a constant $C'>0$ that does not depend upon $n,$ and by the definition of $\tau_n$ in (\ref{tau}), we have
\begin{equation}
\label{conclude}
\int_{[0]} \log\left|\frac{T^{\prime} \circ \Pi \circ \Pi_2}{T^{\prime} \circ \Pi \circ \Pi_1}\right|d\nu\geq \int_{[0]} \log \frac{\tau_n}{C'}d\nu=\nu([0])\log \frac{\tau_n}{C'}=\epsilon \log \frac{\tau_n}{C'}.
\end{equation} Combining \eqref{lyap1} with \eqref{conclude} implies our result.
\\

\textbf{Case $2$ ($T$ is orientation reversing).} We define $$A:=\big\{(b_i) :\min_{i}\{b_i=0\}\textrm{ is even}\big\}$$ and $$B:=\big\{(b_i) :\min_{i}\{b_i=0\}\textrm{ is odd}\big\}.$$In particular
$$A= \bigcup_{w \in \Sigma^{\ast}_{\textnormal{odd}}} [w0]$$
and
$$B= \bigcup_{w \in \Sigma^{\ast}_{\textnormal{even}}} [w0].$$
Where $\Sigma^{\ast}_{\textnormal{odd}}$ denotes all finite words over the alphabet $\mathbb{N}$ of odd length, $\Sigma^{\ast}_{\textnormal{even}}$ denotes all finite words over the alphabet $\mathbb{N}$ of even length. By the Birkhoff ergodic theorem, $\nu\left(\Sigma_0 \setminus A\cup B\right)=0$. 

Since $T$ is orientation reversing and $T'$ is increasing, it follows that for $(b_i) \in A$, $\Pi(\Pi_1((b_i))) \leq \Pi(\Pi_2((b_i))),$ and for $(b_i) \in B$, $\Pi(\Pi_2((b_i)))\leq \Pi(\Pi_1((b_i)))$. Now,
\begin{eqnarray*}
\chi(\mu_{\p})&=& \int_{A} \log|T^{\prime}\circ \Pi \circ \Pi_2|d\nu+\int_{B} \log|T^{\prime}\circ \Pi \circ \Pi_2|d\nu
\end{eqnarray*}
and similarly
\begin{eqnarray*}
\chi(\mu_{\pen})&=& \int_{A} \log|T^{\prime}\circ \Pi \circ \Pi_1|d\nu+\int_{B} \log|T^{\prime}\circ \Pi \circ \Pi_1|d\nu
\end{eqnarray*}
Thus, 
\begin{multline}
\label{separate}
\chi(\mu_{\p})- \chi({\mu_{\pen}})= \int_{B}\log|T^{\prime}\circ \Pi \circ \Pi_2|- \log|T^{\prime}\circ \Pi \circ \Pi_1|d\nu - \\
 \int_{A}\log|T^{\prime}\circ \Pi \circ \Pi_1|- \log|T^{\prime}\circ \Pi \circ \Pi_2|d\nu 
 \end{multline}
Fix $w \in \Sigma^{\ast}_{\textnormal{even}}$ (where $w$ can be the `empty' word). We begin by showing that
\begin{equation}
\label{begin}
\int_{[w0]}\log|T^{\prime}\circ \Pi \circ \Pi_2|- \log|T^{\prime}\circ \Pi \circ \Pi_1|d\nu \geq \int_{\bigcup_{k \in \mathbb{N}} [kw0]}\log|T^{\prime}\circ \Pi \circ \Pi_1|- \log|T^{\prime}\circ \Pi \circ \Pi_2|d\nu 
\end{equation}
Since $\nu$ is $\sigma_0$ invariant, we can rewrite the first integral above as
\begin{eqnarray}
\label{rewrite}
& &\int_{[w0]}\log|T^{\prime}\circ \Pi \circ \Pi_2|- \log|T^{\prime}\circ \Pi \circ \Pi_1|d\nu\circ \sigma_0^{-1} \nonumber\\
 & &= \int_{\bigcup_{k \in \mathbb{N}_0} [kw0]}\log|T^{\prime}\circ \Pi \circ \Pi_2 \circ \sigma_0|- \log|T^{\prime}\circ \Pi \circ \Pi_1\circ \sigma_0|d\nu \nonumber\\
& &= \int_{\bigcup_{k \in \mathbb{N}_0} [kw0]}\log|T^{\prime}\circ T\circ \Pi \circ \Pi_2|- \log|T^{\prime}\circ T\circ \Pi \circ \Pi_1|d\nu.
\end{eqnarray}
Here $\mathbb{N}_0= \mathbb{N}\cup\{0\}.$ The final line follows because $\Pi_1\circ \sigma_0=\sigma\circ \Pi_1$, $\Pi_2\circ \sigma_0=\sigma\circ \Pi_2$ and $\Pi \circ \sigma=T \circ \Pi$. For all $(b_i) \in [0w0]$, $T(\Pi \circ \Pi_2((b_i)))\leq T(\Pi\circ\Pi_1((b_i))),$ therefore since $|T'|$ is decreasing,
\begin{equation}
\label{zero inequality}
\int_{[0w0]}\log|T^{\prime}\circ T\circ \Pi \circ \Pi_2|- \log|T^{\prime}\circ T\circ \Pi \circ \Pi_1|d\nu \geq 0 .
\end{equation}
So to prove \eqref{begin} it is enough to prove
\begin{eqnarray} \label{int2}
\int_{\bigcup_{k \in \mathbb{N}} [kw0]}\log|T^{\prime}\circ T\circ \Pi \circ \Pi_2|- \log|T^{\prime}\circ T\circ \Pi \circ \Pi_1|d\nu \geq \nonumber \\
\int_{\bigcup_{k \in \mathbb{N}} [kw0]}\log|T^{\prime}\circ \Pi \circ \Pi_1|- \log|T^{\prime}\circ \Pi \circ \Pi_2|d\nu 
\end{eqnarray}
If we let $(b_i) \in [kw0]$ and put $x= \Pi\circ \Pi_2((b_i))$, $y=\Pi\circ\Pi_1((b_i))$ we see that $x>y$. We know by property $(2)$ that $(T^2)'(x) \geq (T^2)'(y)$ for any $x>y$ for which $x,y \in \mathcal{I}_k$, therefore by an application of the chain rule we have
$$\log \left|\frac{T^{\prime}\circ T(\Pi \circ \Pi_2((b_i)))}{T^{\prime}\circ T(\Pi \circ \Pi_1((b_i)))}\right| \geq  \log \left| \frac{T^{\prime}(\Pi \circ \Pi_1((b_i)))}{T^{\prime}(\Pi \circ \Pi_2((b_i)))} \right|$$
from which we can deduce \eqref{int2}. Applying the above equations we obtain the following
\begin{eqnarray}
\label{bigarray}
\chi(\mu_{\p})-\chi(\mu_{\pen}) &\stackrel{\eqref{separate}}=& \sum_{w\in \Sigma_{\textnormal{even}}^*}\Big(\int_{[w0]}\log|T^{\prime}\circ \Pi \circ \Pi_2| -\log|T^{\prime}\circ \Pi \circ \Pi_1|d\nu \nonumber\\ &-&\int_{\bigcup_{k \in \mathbb{N}}[kw0]}\log|T^{\prime}\circ \Pi \circ \Pi_1|- \log|T^{\prime}\circ \Pi \circ \Pi_2|d\nu\Big)\,  \nonumber \\
&\stackrel{\eqref{begin}}\geq &\int_{[0]}\log|T^{\prime} \circ \Pi \circ \Pi_2|- \log|T^{\prime}\circ \Pi \circ \Pi_1|d\nu\nonumber\\ &-&\int_{\bigcup_{k \in \mathbb{N}}[k0]}\log|T^{\prime}\circ \Pi \circ \Pi_1|- \log|T^{\prime}\circ \Pi \circ \Pi_2|d\nu\,  \nonumber \\
&\stackrel{\eqref{rewrite}}= &\int_{\bigcup_{k \in \mathbb{N}_0}[k0]}\log|T^{\prime}\circ T\circ  \Pi \circ \Pi_2|- \log|T^{\prime}\circ T \circ \Pi \circ \Pi_1|d\nu\nonumber\\ &-&\int_{\bigcup_{k \in \mathbb{N}}[k0]}\log|T^{\prime}\circ \Pi \circ \Pi_1|- \log|T^{\prime}\circ \Pi \circ \Pi_2|d\nu\, \nonumber\\
&= &\int_{[00]}\log|T^{\prime}\circ T\circ  \Pi \circ \Pi_2|- \log|T^{\prime}\circ T \circ \Pi \circ \Pi_1|d\nu\nonumber\\
&+ &\int_{\bigcup_{k \in \mathbb{N}}[k0]}\log|T^{\prime}\circ T\circ  \Pi \circ \Pi_2|- \log|T^{\prime}\circ T \circ \Pi \circ \Pi_1|d\nu\nonumber\\ &-&\int_{\bigcup_{k \in \mathbb{N}}[k0]}\log|T^{\prime}\circ \Pi \circ \Pi_1|- \log|T^{\prime}\circ \Pi \circ \Pi_2|d\nu\,\nonumber\\
&= &\int_{[00]}\log\Big|\frac{T^{\prime}\circ T\circ  \Pi \circ \Pi_2}{T^{\prime}\circ T \circ \Pi \circ \Pi_1}\Big|\\
&+& \int_{\bigcup_{k \in \mathbb{N}} [k0]}\log\left|\frac{T^{\prime}\circ T\circ \Pi \circ \Pi_2 \cdot T^{\prime}\circ \Pi \circ \Pi_2}{T^{\prime}\circ T\circ \Pi \circ \Pi_1\cdot T^{\prime}\circ \Pi \circ \Pi_1}\right|d\nu(\i)d\nu\nonumber .
\end{eqnarray}
If $(b_i)\in [k0]$ for some $k\in\mathbb{N}_0,$ then $T\circ \Pi \circ \Pi_1((b_i))\in I_1,$ and $T\circ \Pi \circ \Pi_2((b_i))\in I_n.$ Repeating the argument given in Case $1,$ we can assert that there exists $D>0$ such that 
\begin{equation}
\label{derivativefrac1}
\frac{T^{\prime}\circ T\circ \Pi \circ \Pi_2((b_i))}{T^{\prime}\circ T\circ \Pi \circ \Pi_1((b_i))}\geq D\tau_n.
\end{equation}
If $(b_i)\in [k0]$ for some $k\in\mathbb{N}$ then $\Pi \circ \Pi_1((b_i)),\Pi \circ \Pi_2((b_i))\in I_k.$ In which case it follows from Lemma \ref{bounded distortion} that there exists $D'>0$ such that for all $(b_i)\in[k0],$ we have  
\begin{equation}
\label{derivativefrac2}
\frac{T^{\prime}\circ \Pi \circ \Pi_2((b_i))}{T^{\prime}\circ \Pi \circ \Pi_1((b_i))}\geq D'.
\end{equation}
Without loss of generality we may assume that $D'<1$. In which case substituting \eqref{derivativefrac1} and \eqref{derivativefrac2} into \eqref{bigarray} we obtain
\begin{align*}
\chi(\mu_{\p})-\chi(\mu_{\pen})&\geq \int_{[00]}\log DD'\tau_n d\nu(i)+ \int_{\bigcup_{k \in \mathbb{N}} [k0]}\log DD'\tau_n d\nu(i)\\
&= \Big(\nu([00])+\nu(\cup_{k\in\mathbb{N}}[k0]))\Big)\log DD'\tau_n\\
&=\epsilon \log DD'\tau_n.
\end{align*}Which completes our proof.

\end{proof}

\subsection{Maximising measures supported on finitely many symbols}
Fix a $T$ satisfying properties $(1)-(5)$ and $L\in\mathbb{N}.$ One can consider the sequence space $\Sigma_L:=\{1,\ldots,L\}^{\mathbb{N}}$ and the projection map $\Pi:\Sigma_L\to [0,1]$ given by restricting $\Pi$ to $\Sigma_L.$ By a small abuse of notation we will also denote this restricted map by $\Pi$. Given a probability vector $\p=(p_i)_{i=1}^{L},$ we can consider the corresponding Bernoulli measure $m_{\p}$ supported on $\Sigma_{L}$, and its associated pushforward $\mu_\p$. Just as in the case of infinitely many digits, one can ask whether there exists a Bernoulli measure whose pushforward has maximal dimension. Importantly when we restrict to finitely many digits, the set of Bernoulli measures $m_\p$ is now compact with respect to the weak star topology, and the map $m_\p \to \dim \mu_{\p}$ is continuous by Lemma \ref{volume lemma}. Therefore there must exist $\p^L$ such that $$\dim \mu_{\p^L}= \sup_{\p=(p_i)_{i=1}^L}\dim \mu_{\p}.$$ Note that we can apply Lemma \ref{volume lemma} because any measure $\mu_\p$ supported on at most $L$ digits always satisfies $h(\mu_\p)\leq \log L$. To construct the maximising measure $\mu_{\p^*}$ whose existence is asserted by Theorem \ref{Main theorem}, we will make use of the sequence of measures $(\mu_{\p^L})$. 

We now introduce a useful class of measures that exhibit a certain decay property. Let $C>0$ and $\alpha\in(s_0,1),$ we say that a probability vector $\p$ exhibits $(C,\alpha)$ decay if $$p_i\leq \frac{C}{\tau_i^{\alpha}} \textrm{ for all }i\in\mathbb{N}.$$ Recall here that $\tau_n:=\inf_{x\in I_n}|T'(x)|.$ We let $$D(C,\alpha):=\{\mu_\p:\p \textrm{ exhibits }(C,\alpha) \textrm{ decay}\}.$$
In the following proposition we show that when $h(\mu_\p)<\infty$, the quantity $\dim_{H}(\mu_\p)$ can be approximated arbitrarily well by the dimension of a measure supported on finitely many symbols. We also prove that $\dim(\cdot)$ is continuous as a real valued function on $D(C,\alpha)$ with respect to the weak star topology. 

\begin{proposition}
	\label{long prop}Assume $T$ satisfies properties $(1)-(5)$. The following properties hold:
	\begin{enumerate}
		\item Suppose $\mu_{\p}$ satisfies $h(\mu_\p)<\infty$. Then for any $\epsilon>0$ there exists $L\in\mathbb{N}$ and a measure $\mu_{\p,L}$ supported on $L$ symbols such that $$|\dim  \mu_{\p}- \dim  \mu_{\p,L}|<\epsilon.$$ 
		\item $\dim:D(C,\alpha)\to\mathbb{R}$ is continuous.
	\end{enumerate}
\end{proposition}
\begin{proof}
We start by proving $(1)$. Without loss of generality we may assume that $\dim \mu_{\p}$ is strictly positive, and so by Lemma \ref{volume lemma} we must also have $\chi(\mu_\p)<\infty$. Otherwise we may simply take $\mu_{\p,L}$ to be a measure supported on a single point and our proof is complete.   

It follows from the fact that $\mu_\p$ is $T$-invariant, the chain rule, and Lemma \ref{volume lemma}, that for any $n\in\mathbb{N}$ we have 
 \begin{equation}
 \label{nfold}\dim \mu_{\p}= \frac{-\sum_{\mathbf{j}\in \mathbb{N}^n} m_\p ([\mathbf{j}])\log m_\p([\mathbf{j}])}{\int \log |(T^n)'|\, d\mu_\p}.
 \end{equation} Let us now fix $\epsilon>0$. It follows from our uniformly expanding assumption that $\inf_{x\in[0,1]} |(T^n)'(x)|$ becomes arbitrarily large as $n\to\infty$. Combining this observation with Lemma \ref{bounded distortion} and \eqref{nfold}, it follows that we can pick $n$ sufficiently large that for any $\p'$ satisfying $h(\mu_{\p'})<\infty$ and $\chi(\mu_\p')<\infty$ we have 
 \begin{equation}
 \label{1} 
 \Big|\dim \mu_{\p'}-\frac{-\sum_{\mathbf{j}\in \mathbb{N}^n} m_{\p'} ([\mathbf{j}])\log m_{\p'}([\mathbf{j}])}{\sum_{\mathbf{j}\in \mathbb{N}^n} m_{\p'}([\mathbf{j}])\log |(T^n)'(x_\mathbf{j})|}\Big|<\frac{\epsilon}{4}.
 \end{equation}
 Where $x_\mathbf{j}$ is an arbitrary element of $I_\mathbf{j}$. Let us now pick $L\in\mathbb{N}$ sufficiently large that 
 \begin{equation}
 \label{2}\Big|\frac{-\sum_{\mathbf{j}\in \{1,\ldots,L\}^n} m_\p ([\mathbf{j}])\log m_\p([\mathbf{j}])}{\sum_{\mathbf{j}\in \{1,\ldots,L\}^n} m_{\p}([\mathbf{j}])\log |(T^n)'(x_\mathbf{j})|}- \frac{-\sum_{\mathbf{j}\in \mathbb{N}^n} m_\p ([\mathbf{j}])\log m_\p([\mathbf{j}])}{\sum_{\mathbf{j}\in \mathbb{N}^n} m_{\p}([\mathbf{j}])\log |(T^n)'(x_\mathbf{j})|}\Big|<\frac{\epsilon}{4}
 \end{equation}and 
 \begin{equation}
 \label{A}
 \Big|\log \Big(\sum_{i=1}^L m_\p ([i])\Big)^n\Big|<\frac{\epsilon}{4}.
 \end{equation} Let $m_{\p,L}$ be the Bernoulli measure defined via the equation 
\[ m_{\p,L}([j]): = \left\{ \begin{array}{ll}
\frac{m_{\p}([j])}{\sum_{i=1}^L m_\p ([i])} & \mbox{if $1\leq j \leq L$};\\
0 & \mbox{if $j>L$}.\end{array} \right. \] 
We observe 
\begin{align*}
&\frac{-\sum_{\mathbf{j}\in \{1,\ldots,L\}^n} m_\p ([\mathbf{j}])\log m_\p([\mathbf{j}])}{\sum_{\mathbf{j}\in \{1,\ldots,L\}^n} m_{\p}([\mathbf{j}])\log |(T^n)'(x_\mathbf{j})|}\\
=&\frac{-\sum_{\mathbf{j}\in \{1,\ldots,L\}^n} m_{\p,L} ([\mathbf{j}])\log m_{\p}([\mathbf{j}])}{\sum_{\mathbf{j}\in \{1,\ldots,L\}^n} m_{\p,L}([\mathbf{j}])\log |(T^n)'(x_\mathbf{j})|}\\
=&\frac{-\sum_{\mathbf{j}\in \{1,\ldots,L\}^n} m_{\p,L} ([\mathbf{j}])\log m_{\p}([\mathbf{j}])+\sum_{\mathbf{j}\in \{1,\ldots,L\}^n} m_{\p,L} ([\mathbf{j}])\log m_{\p,L}([\mathbf{j}])}{\sum_{\mathbf{j}\in \{1,\ldots,L\}^n} m_{\p,L}([\mathbf{j}])\log |(T^n)'(x_\mathbf{j})|}\\
-&\frac{\sum_{\mathbf{j}\in \{1,\ldots,L\}^n} m_{\p,L} ([\mathbf{j}])\log m_{\p,L}([\mathbf{j}])}{\sum_{\mathbf{j}\in \{1,\ldots,L\}^n} m_{\p,L}([\mathbf{j}])\log |(T^n)'(x_\mathbf{j})|}\\
=& -\frac{\log \Big(\sum_{i=1}^L m_\p ([i])\Big)^n}{\sum_{\mathbf{j}\in \{1,\ldots,L\}^n} m_{\p,L}([\mathbf{j}])\log |(T^n)'(x_\mathbf{j})|}-\frac{\sum_{\mathbf{j}\in \{1,\ldots,L\}^n} m_{\p,L} ([\mathbf{j}])\log m_{\p,L}([\mathbf{j}]) }{\sum_{\mathbf{j}\in \{1,\ldots,L\}^n} m_{\p,L}([\mathbf{j}])\log |(T^n)'(x_\mathbf{j})|}. 
\end{align*}
Applying \eqref{A} we have  
\begin{equation}
\label{3}
\Big|\frac{-\sum_{\mathbf{j}\in \{1,\ldots,L\}^n} m_\p ([\mathbf{j}])\log m_\p([\mathbf{j}])}{\sum_{\mathbf{j}\in \{1,\ldots,L\}^n} m_{\p}([\mathbf{j}])\log |(T^n)'(x_\mathbf{j})|}- \frac{-\sum_{\mathbf{j}\in \{1,\ldots,L\}^n} m_{\p,L} ([\mathbf{j}])\log m_{\p,L}([\mathbf{j}])}{\sum_{\mathbf{j}\in \{1,\ldots,L\}^n} m_{\p,L}([\mathbf{j}])\log |(T^n)'(x_\mathbf{j})|}\Big|< \frac{\epsilon}{4}.
\end{equation}
Now by repeated applications of the triangle inequality we obtain
\begin{align}
\label{bounding}
&|\dim(\mu_{\p})-\dim(\mu_{\p,L})|\nonumber \\
\leq &\Big|\dim(\mu_{\p})-\frac{-\sum_{\mathbf{j}\in \mathbb{N}^n} m_\p ([\mathbf{j}])\log m_\p([\mathbf{j}])}{\sum_{\mathbf{j}\in \mathbb{N}^n} m_{\p}([\mathbf{j}])\log |(T^n)'(x_\mathbf{j})|}\Big|\\
+ &\Big|\frac{-\sum_{\mathbf{j}\in \mathbb{N}^n} m_\p ([\mathbf{j}])\log m_\p([\mathbf{j}])}{\sum_{\mathbf{j}\in \mathbb{N}^n} m_{\p}([\mathbf{j}])\log |(T^n)'(x_\mathbf{j})|}-\frac{-\sum_{\mathbf{j}\in \{1,\ldots,L\}^n} m_\p ([\mathbf{j}])\log m_\p([\mathbf{j}])}{\sum_{\mathbf{j}\in \{1,\ldots,L\}^n} m_{\p}([\mathbf{j}])\log |(T^n)'(x_\mathbf{j})|}\Big|\nonumber\\
+&\Big|\frac{-\sum_{\mathbf{j}\in \{1,\ldots,L\}^n} m_\p ([\mathbf{j}])\log m_\p([\mathbf{j}])}{\sum_{\mathbf{j}\in \{1,\ldots,L\}^n} m_{\p}([\mathbf{j}])\log |(T^n)'(x_\mathbf{j})|}-\frac{-\sum_{\mathbf{j}\in \{1,\ldots,L\}^n} m_{\p,L} ([\mathbf{j}])\log m_{\p,L}([\mathbf{j}])}{\sum_{\mathbf{j}\in \{1,\ldots,L\}^n} m_{\p,L}([\mathbf{j}])\log |(T^n)'(x_\mathbf{j})|}\Big|\nonumber\\
+&\Big|\frac{-\sum_{\mathbf{j}\in \{1,\ldots,L\}^n} m_{\p,L} ([\mathbf{j}])\log m_{\p,L}(\mathbf{j})}{\sum_{\mathbf{j}\in \{1,\ldots,L\}^n} m_{\p,L}([\mathbf{j}])\log |(T^n)'(x_\mathbf{j})|}-\dim(\mu_{\p,L})\Big|\nonumber\\
< &4\cdot \frac{\epsilon}{4}=\epsilon\nonumber.
\end{align}Where we bounded the terms in \eqref{bounding} in order by, \eqref{1}, \eqref{2}, \eqref{3}, and \eqref{1} respectively. This completes our proof of item $(1)$.

We now prove item $(2)$. Let us fix $C>0$ and $\alpha\in(s_0,1)$.  It follows from the definition of $\p$ exhibiting $(C,\alpha)$ decay that for any $n\in\mathbb{N}$ and $\epsilon>0$, we can pick $L\in\mathbb{N}$ depending only on $n,C,$ and $\alpha$, such that:
\begin{enumerate}
	\item For all $\mu_\p\in D(C,\alpha)$ $$\Big|\sum_{\mathbf{j}\in\mathbb{N}^n}m_\p([\mathbf{j}])\log m_\p([\mathbf{j}])-\sum_{\mathbf{j}\in\{1,\ldots,L\}^n}m_\p([\mathbf{j}])\log m_\p([\mathbf{j}])\Big|<\epsilon$$
	\item For all $\mu_\p\in D(C,\alpha)$ $$\Big|\sum_{\mathbf{j}\in\mathbb{N}^n}m_\p([\mathbf{j}])\log|(T^n)'(x_\mathbf{j})|-\sum_{\mathbf{j}\in\{1,\ldots,L\}^n}m_\p([\mathbf{j}]) \log|(T^n)'(x_\mathbf{j})|\Big|<\epsilon.$$
\noindent Similarly, it follows from the definition of $\p$ exhibiting $(C,\alpha)$ decay that for any $n\in\mathbb{N}$, one can pick $K,K'>0$ depending only on $n,C,$ and $\alpha$, such that: 
\\
\item For all $\mu_\p\in D(C,\alpha)$ $$-\sum_{\mathbf{j}\in\mathbb{N}^n}m_\p([\mathbf{j}])\log m_\p([\mathbf{j}])<K$$ 

	\item For all $\mu_\p\in D(C,\alpha)$ $$\sum_{\mathbf{j}\in\mathbb{N}^n}m_\p([\mathbf{j}]) \log|(T^n)'(x_\mathbf{j})|<K'.$$ 
	\end{enumerate}

Applying the above properties, we can assert that for any $\epsilon>0,$ there exists $L\in\mathbb{N}$ depending only on $n,C$ and $\alpha$, such that for any $\mu_\p\in D(C,\alpha)$ we have
\begin{equation*}
\Big|\frac{-\sum_{\mathbf{j}\in \{1,\ldots,L\}^n} m_\p ([\mathbf{j}])\log m_\p([\mathbf{j}])}{\sum_{\mathbf{j}\in \{1,\ldots,L\}^n} m_{\p}([\mathbf{j}])\log |(T^n)'(x_\mathbf{j})|}- \frac{-\sum_{\mathbf{j}\in \mathbb{N}^n} m_\p ([\mathbf{j}])\log m_\p([\mathbf{j}])}{\sum_{\mathbf{j}\in \mathbb{N}^n} m_{\p}([\mathbf{j}])\log |(T^n)'(x_\mathbf{j})|}\Big|<\epsilon.
\end{equation*} Therefore, by \eqref{1} we can assert that for any $\epsilon>0$, there exists $n\in\mathbb{N},$ and $L\in\mathbb{N}$ depending only on $n,C,$ and $\alpha,$ such that for any $\mu_\p\in D(C,\alpha)$ we have
\begin{equation}
\label{Lbound}
\Big|\frac{-\sum_{\mathbf{j}\in \{1,\ldots,L\}^n} m_\p ([\mathbf{j}])\log m_\p([\mathbf{j}])}{\sum_{\mathbf{j}\in \{1,\ldots,L\}^n} m_{\p}([\mathbf{j}])\log |(T^n)'(x_\mathbf{j})|}- \dim \mu_\p \Big|<\epsilon.
\end{equation} 
Now let us a fix a measure $\mu_\p\in D(C,\alpha)$ and assume $(\mu_{\p_k})$ is a sequence in $D(C,\alpha)$ such that $\mu_{\p_k}\to \mu_{\p}$ with respect to the weak star topology. Then for any $\epsilon>0$ we have
\begin{align*}
&\lim_{k\to\infty}|\dim(\mu_{\p_k})- \dim(\mu_{\p})|\\
&\leq \lim_{k\to\infty}\Big|\dim(\mu_{\p_k})-\frac{-\sum_{\mathbf{j}\in \{1,\ldots,L\}^n} m_{\p_k} ([\mathbf{j}])\log m_{\p_k}([\mathbf{j}])}{\sum_{\mathbf{j}\in \{1,\ldots,L\}^n} m_{\p_k}([\mathbf{j}])\log |(T^n)'(x_\mathbf{j})|}\Big|\\
&+\Big|\frac{-\sum_{\mathbf{j}\in \{1,\ldots,L\}^n} m_{\p_k} ([\mathbf{j}])\log m_{\p_k}([\mathbf{j}])}{\sum_{\mathbf{j}\in \{1,\ldots,L\}^n} m_{\p_k}([\mathbf{j}])\log |(T^n)'(x_\mathbf{j})|}-\frac{-\sum_{\mathbf{j}\in \{1,\ldots,L\}^n} m_\p ([\mathbf{j}])\log m_\p([\mathbf{j}])}{\sum_{\mathbf{j}\in \{1,\ldots,L\}^n} m_{\p}([\mathbf{j}])\log |(T^n)'(x_\mathbf{j})|}\Big|\\
&+\Big|\frac{-\sum_{\mathbf{j}\in \{1,\ldots,L\}^n} m_\p ([\mathbf{j}])\log m_\p([\mathbf{j}])}{\sum_{\mathbf{j}\in \{1,\ldots,L\}^n} m_{\p}([\mathbf{j}])\log |(T^n)'(x_\mathbf{j})|}- \dim(\mu_{\p})\Big|\\
&<2\epsilon.
\end{align*} Here $n\in\mathbb{N}$ and $L\in\mathbb{N}$ were chosen so \eqref{Lbound} holds. We used weak star convergence to conclude that the second term converges to zero. Since $\epsilon$ is arbitrary we have $\lim_{k\to\infty}\dim(\mu_{\p_k})= \dim(\mu_{\p})$ as required.
\end{proof}

The following proposition gives conditions guaranteeing that our maximising measures on $L$ symbols, the $\mu_{\p^L},$ are contained in $D(C,\alpha)$ for appropriate choices of $C>0$ and $\alpha$. 

\begin{proposition}
	\label{decay prop}
Suppose there exists $\mu_{\p}$ such that $\dim \mu_{\p}>s_0.$ Then there exists $C>0$ and $\alpha\in(s_0,1)$ such that $\mu_{\p^L}\in D(C,\alpha)$ for all $L\in\mathbb{N}$. 
\end{proposition}
\begin{proof}
It suffices to prove the result for $L$ sufficiently large. By item $(1)$ from Proposition \ref{long prop} there must exist $\mu_{\p'}$ supported on finitely many digits such that $\dim \mu_{\p'}>s$ for some $s>s_0$. Since $\dim \mu_{\p^L}$ is increasing with $L,$ we may therefore assume that for all $L\in\mathbb{N}$ sufficiently large we have $\dim \mu_{\p^L}>s.$ 

Let $\kappa>1$ and $\alpha\in(s_0,1)$ be sufficiently small so $\dim \mu_{\p^L}>\kappa\alpha$ for all $L$ sufficiently large. Now let us assume for a contradiction that there does not exist $C>0$ such that $\mu_{\p^L} \in D(C,\alpha)$ for all $L \in \mathbb{N}$. Therefore there must exist $L\in\mathbb{N}$ arbitrarily large, and $n$ arbitrarily large, such that $1\leq n \leq L$ and 
\begin{equation}
\label{faildecay}
p^L_n> \frac{1}{\tau_n^{\alpha}}.
\end{equation}Let us now fix such $L$ and $n$. We consider the measure $\mu_{\p^{L}_{\epsilon,n}}$ obtained when we transfer $\epsilon$ of the mass from the digit $n$ to the digit $1$. Then by Lemma \ref{volume lemma}, Lemma \ref{entropy}, and Lemma \ref{lyapunov}, for sufficiently small $\epsilon>0,$ and $L$ and $n$ sufficiently large, 
\begin{eqnarray}
\dim \mu_{\p^{L}_{\epsilon,n}} &=& \frac{h(\mu_{\p^{L}_{\epsilon,n}})}{\chi(\mu_{\p^{L}_{\epsilon,n}})}\nonumber\\
&\geq&\frac{h(\mu_{\p^{L}})-\epsilon \kappa \log(\frac{p_1}{p_n})}{\chi(\mu_{\p^{L}})-\epsilon\log (\lambda \tau_n)} \nonumber \\
&\stackrel{\eqref{faildecay}}\geq& \frac{h(\mu_{\p^{L}})-\epsilon \kappa \log( \tau_n^{\alpha})}{\chi(\mu_{\p^{L}})-\epsilon\log (\lambda \tau_n)} \nonumber \\
&>& \frac{h(\mu_{\p^{L}})}{\chi(\mu_{\p^{L}})} \label{increase}.
\end{eqnarray}
Where (\ref{increase}) follows because (now abbreviating $h=h(\mu_{\p^{L}})$ and $\chi(\mu_{\p^{L}})=\chi$):
$$	\frac{h-\epsilon\kappa \log( \tau_n^{\alpha})}{\chi-\epsilon \log(\lambda \tau_n)} > \frac{h}{\chi} \Leftrightarrow	\frac{h}{\chi}> \frac{\kappa\alpha \log \tau_n}{\log \lambda +\log \tau_n},$$
and for $L$ and $1\leq n\leq L$ sufficiently large
$$\frac{h}{\chi}= \dim \mu_{\p^{L}} >\frac{\kappa\alpha \log \tau_n}{\log \lambda +\log \tau_n}.$$
In the last inequality we used the fact that for sufficiently large $n$ the right hand side resembles $\kappa\alpha,$ and by definition $\dim \mu_{\p^{L}}>\kappa\alpha$. Therefore $\mu_{\p^{L}_{\epsilon,n}}$ is another measure supported on the first $L$ symbols such that $\dim \mu_{\p^{L}_{\epsilon,n}}> \dim \mu_{\p^{L}}.$ This contradicts the fact $\dim \mu_{\p^L}$ is maximal and our result follows.

\end{proof}

We are now in a position to prove Theorem \ref{Main theorem}.

\begin{proof}[Proof of Theorem \ref{Main theorem}]
By assumption there exists $\mu_{\p}$ such that $\dim \mu_{\p}>s_0$. Applying Proposition \ref{decay prop} we know that the maximising measures $\mu_{\p^L}$ are contained in $D(C,\alpha)$ for some $C>0$ and $\alpha\in(s_0,1)$. Let
\begin{align*}
\p^1&=(1,0,0\ldots,)\\
\p^2&=(p_1^2,p_2^2,0,0\ldots,)\\
&\cdots\\
\p^L&=(p_1^L,p_2^L,\ldots,p_L^L,0,0\ldots).
\end{align*} 
By considering subsequences we can assert that there exists a vector $\p^*=(p_i^*)_{i=1}^{\infty}$, such that $p_i^{L_n}\to p_i^*$ along some subsequence $(L_n)$ for all $i\in\mathbb{N}$. Note that $\p^*$ is a probability vector, and the corresponding measure $\mu_{\p^{*}}$ is contained in $D(C,\alpha)$. This is a consequence of $\mu_{\p^{L}}$ being contained $D(C,\alpha)$ for all $L\in\mathbb{N}$. 

By the above $\mu_{\p^{L_n}}\to \mu_{\p^*}$ in the weak star topology. We claim that  
\begin{equation}
\label{claim}
\sup_{\p} \dim \mu_\p= \dim \mu_{\p^*}.
\end{equation} To see this, suppose that \eqref{claim} is not true and there exists $\p'$ such that $\dim (\mu_{\p'})> \dim \mu_{\p^*}.$ It follows from Proposition \ref{long prop} that there must exist $\mu_\p''$ supported on finitely many symbols such that $\dim \mu_\p''> \dim \mu_{\p^*}.$ Since $\dim \mu_{\p^{L+1}}\geq \dim \mu_{\p^L}$ for all $L\in\mathbb{N},$ we can pick $N\in\mathbb{N}$ sufficiently large such that for all $n\geq N$ we have $\dim \mu_{\p^{L_n}}\geq \dim (\mu_\p'').$ However, by Proposition \ref{long prop} we know that $\lim_{n\to\infty}\dim \mu_{\p^{L_n}}=\dim \mu_{\p^*}$. Which is impossible given  $\dim \mu_{\p^{L_n}}\geq \dim (\mu_\p'')$. So we have our contradiction and \eqref{claim} must hold.
\end{proof}

\section{Proof of Theorem \ref{supremum cf} and final comments}
To use Theorem \ref{Main theorem} to prove Theorem \ref{supremum cf}, we need to demonstrate that there exists a measure $\mu_{\p}$ satisfying $\dim \mu_{\p}>s_0$. We remark that for the Gauss map we have $s_0=1/2$. Consider the probability vector $\p_{1/3}=(1/3,1/3,1/3,0\ldots).$ It is straightforward to show that $h(\mu_{\p_{1/3}})=\log 3$. Since $\p_{1/3}$ is supported on $3$ digits and $\log T'(x)=-2\log x$ we have \begin{align*}
\int \log |T'|d\mu_{\p_{1/3}}=\sum_{1\leq i,j\leq 3} \int_{I_{ij}} \log |T'|d\mu_{\p_{1/3}}&\leq \sum_{1\leq i,j\leq 3}\int_{I_{ij}} \max_{x\in I_{ij}}-2\log x\, d\mu_{\p_{1/3}}\\
&=\frac{2}{9}\sum_{1\leq i,j\leq 3} \max_{x\in I_{ij}}\log \frac{1}{x}.
\end{align*}Importantly this last term lends itself to explicit calculation. Performing the relevant calculations it can be shown that 
$$\int \log |T'|d\mu_{\p_{1/3}}\leq 1.79811\ldots.$$ Consequently, using Lemma \ref{volume lemma} we have $$\dim \mu_{\p_{1/3}}=\frac{h(\mu_{\p_{1/3}})}{\chi(\mu_{\p_{1/3}})}\geq \frac{\log 3}{1.79811}\approx 0.611\ldots.$$ So $\dim \mu_{\p_{1/3}}>1/2$ as required. By Theorem \ref{Main theorem} we may conclude Theorem \ref{supremum cf}. 

\begin{remark}
Theorem \ref{Main theorem} gives conditions guaranteeing the existence of a measure $\mu_{\p}$ whose dimension is maximal amongst the class of pushforwards of Bernoulli measures. This doesn't immediately imply the existence of a dimension gap at $1$, it merely reduces the question to showing that 
\begin{equation}
\label{weak gap}
\dim \mu_{\p}<1 \textrm{ for all } \mu_\p.
\end{equation}
Note that for the Luroth map there exists a Bernoulli measure whose pushforward is the Lebesgue measure restricted to $[0,1]$. If there is a Bernoulli measure $\p$ whose pushforward satisfies $\dim \mu_{\p}=1$, then by a result of Walters \cite{Wal}, we know that it is the unique absolutely continuous $T$-invariant measure. One can verify the Gauss measure is not the pushforward of a Bernoulli measure via explicit calculation. Using the closed formula for the density one can show that $\mu_{G}(I_{12})\neq \mu_{G}(I_{21}).$ Therefore \eqref{weak gap} holds and we have a dimension gap at $1$. For a general $T,$ one might not necessarily have a nice closed form for the density. In this case one can appeal to a cohomological argument. If there exists a $\p$ such that $\dim \mu_{\p}=1,$  then it follows from \cite[Theorem 16]{Wal} that any point satisfying $T^n(x)=x$ must also satisfy
$$-\log |(T^n)'(x)|=\sum_{i=0}^{n-1}\phi((T^i)(x)),$$ where $\phi(x)=\log p_i$ if $p_i>0$ and $x\in I_i$, and $\phi(x)=0$ if $p_i=0$ and $x \in I_i$. We remark that the right hand side of the above does not depend on the order of the $I_i$ visited by $x.$ Consequently, if $y$ is another point such that $T^{n}(y)=y,$ and the orbit of $y$ visits the same intervals as $x,$ then we must have $|(T^n)'(x)|=|(T^n)'(y)|.$ If we have an explicit formula for $T,$ one can hope to verify whether $|(T^n)'(x)|=|(T^n)'(y)|$ for all such $x$ and $y$. If we can find such an $x$ and $y$ satisfying $|(T^n)'(x)|\neq |(T^n)'(y)|,$ it would follow that \eqref{weak gap} holds for all $\mu_{\p}$, and so we must have a dimension gap at $1$.

\end{remark}

\begin{example}
	As an example to illustrate the above remark, consider $\{I_n\}_{n=1}^{\infty},$ where $I_1=(0,1/2)$ and $T_1(x)=2x$, $I_2=(1/2,3/4)$ and $T_2(x)=4x\mod 1,$ and $I_3=(3/4,0.861\ldots)$ and $T_3(x)=8x + \tan(x-3/4)\mod 1.$ For $n \geq 4$ we assume $T_n$ is the unique affine orientation preserving map sending $I_n$ to $(0,1)$. We can choose $I_n$ in such a way that properties $(1)$--$(5)$ hold and the corresponding $s_0$ can be made arbitrarily small. Moreover, by considering just $T_1$ and $T_2$ we can construct a Bernoulli measure whose pushforward has positive dimension. Therefore we can assume that $I_n$ have been chosen in such a way that the hypothesis of Theorem \ref{Main theorem} is satisfied. 
	
	Consider $x'\approx 0.817 $ such that $(T_1\circ T_2\circ T_3)(x')=x'$ and $y'\approx 0.789$ such that $(T_2\circ T_1\circ T_3)(y')=y'.$ Both $x'$ and $y'$ are periodic points whose orbits visit the same intervals, albeit it in a different order. Performing the relevant calculations we can show that $(T^3)'(x')\approx 72.036$ and $(T^3)'(y')\approx 72.012.$ By the above remark it follows that $\dim \mu_{\p}<1$ for any $\p.$ Applying Theorem \ref{Main theorem} we may deduce that there is a uniform dimension gap at $1$.
\end{example}

\begin{remark}
Theorem \ref{Main theorem} can be used in a general setting to determine the existence of a dimension gap at $1$. However, without knowing there exists a measure $\mu_{\p}$ such that $\dim \mu_{\p}>s_0,$ we cannot apply Theorem \ref{Main theorem} to determine the existence of a Bernoulli measure whose dimension is maximal. It would be interesting to determine a general condition by which one could establish the existence of such a $\mu_{\p}$. We remark that for a map $T$ satisfying $(1)-(5)$ we can construct $\mu_{\p}$ whose dimension can be made arbitrarily close to $s_0$ from below. 
\end{remark}

\begin{remark}
Theorem \ref{supremum cf} and Theorem \ref{Main theorem} establish the existence of a Bernoulli measure whose pushforward has maximal dimension. It is natural to wonder whether this measure is unique. We believe it is unique, however we are unable to prove it.
\end{remark}

\noindent \textbf{Acknowledgements.} The first author was supported by EPSRC grant EP/M001903/1. Part of this work was completed whilst the authors were visiting the Mittag-Leffler institute as part of the program ``Fractal geometry and Dynamics". The authors thank the organisers and staff for their support.


\begin{thebibliography}{1}
\bibitem{FLM} A. H. Fan, L. Liao, J. H. Ma, \textit{On the frequency of partial quotients of regular continued fractions, }
Math. Proc. Cambridge Philos. Soc. 148 (2010), no. 1, 179--192. 
\bibitem{Jur} N. Jurga, \textit{Dimension of Bernoulli measures for non-linear countable Markov maps,} to appear.\bibitem{KPW} Y. Kifer, Y. Peres, B. Weiss, \textit{A dimension gap for continued fractions with independent digits,} Israel J. Math. 124 (2001), 61--76. 
\bibitem{KP} J. Kinney, T. Pitcher, \textit{The dimension of some sets defined in terms of $f$-expansions,} Z. Wahrscheinlichkeitstheorie und Verw. Gebiete 4 1965/1966 293--315.
\bibitem{MU} D. Mauldin, M. Urbanski, \textit{Graph directed Markov systems: 
Geometry and dynamics of limit sets.} Cambridge Tracts in Mathematics, 148. Cambridge University Press, Cambridge, 2003. xii+281 pp. ISBN: 0-521-82538-5.
\bibitem{Rap} A. Rapaport, \textit{A dimension gap for continued fractions with independent digits - the non stationary case,}  to appear in Israel J. Math.
\bibitem{Wal} P. Walters, \textit{Invariant measures and equilibrium states for some mappings which expand distances,} Trans. Amer. Math. Soc. 236 (1978), 121--153. 
\end{thebibliography}
\end{document}